\documentclass{amsart}
\usepackage{amsmath}
\usepackage{amssymb}
\usepackage{amsthm}
\usepackage{epsfig}

\newcommand{\amsprimary}[1]{{\footnotesize\noindent AMS 2010 \textit{Mathematics subject
classification:} Primary #1\vspace{1pc}}}
\newcommand{\keywordsnames}[1]{{\footnotesize\noindent\textit{Key words:} #1\vspace{1pc}}}

\newtheorem{theorem}{Theorem}[section]
\newtheorem{lemma}{Lemma}[section]
\newtheorem{proposition}{Proposition}[section]
\newtheorem{corollary}{Corollary}[section]

\newtheorem{remark}{Remark}[section]

\title[Stability $m$-fold circles]{On the stability of $m$-fold cicles and
the dynamics of generalized 
curve shortening flows}

\begin{document}
\author{
Jean C. Cortissoz
\and
Alexander Murcia
}
\address{Universidad de los Andes, Bogot\'a DC, COLOMBIA.}
\date{}
\begin{abstract}
In this paper we study the (asymptotic and exponential) stability
of the $m$-fold circle as a solution of the $p$-curve shortening flow ($p\geq 1$ an integer).
\end{abstract}
\maketitle
{\keywordsnames {curve-shortening; symmetry; blow-up; exponential convergence.}}

{\amsprimary {54C44; 35K55.}}

\section{Introduction}

Let 
\[
x:\,\mathbb{S}^1\times\left[0,T\right)\longrightarrow \mathbb{R}^2
\]
be a family of smooth immersions of $\mathbb{S}^1$, the unit circle, into $\mathbb{R}^2$.
In this paper we will say that $x$ satisfies the $p$-curve shortening flow, $p\geq 1$, if
$x$ satisfies
\begin{equation}
\label{generalflow}
\frac{\partial x}{\partial t}=-\frac{1}{p}k^{p}N,
\end{equation}
where $k$ is the curvature of the embedding and $N$ is the normal vector pointing outwards the
region bounded by $x\left(\cdot,t\right)$.

Much is known about this family of flows. To select a few among many beautiful and fundamental 
works on the subject, we must mention the works of Gage and Hamilton (\cite{GageHamilton}),
and of Ben Andrews (\cite{Andrews,Andrews2}).

In this paper we will be concerned with the stability of $m$-fold circles as solutions to the 
$p$-curves shortening flow. As it is well known there are small perturbations of the $2$-fold
circle that do not behave asymptotically as a shrinking $2$-fold circle. However,
very recently, Wang in \cite{Wang} showed the asymptotic stability of $m$-fold circles
under certain small $\frac{2\pi m}{n}$-periodic perturbations as solutions
to the curve shortening flow (i.e., for the case when $p=1$). In this note we will extend the 
work of Wang in two ways:  we will show asymptotic stability results for the $m$-fold circle as a solution
to the $p$-curve shortening flow for $p$ any positive integer,
and we will provide sharp stabilization estimates for the
curvature of solutions to the $p$-curve shortening flow that
are appropriate small perturbations of an $m$-fold circle. Other interesting works, 
besides Wang's, regarding
stability of solutions to the curve shortening flow are the by now classical 
papers of Abresch and Langer (\cite{AbreschLanger}), and of Epstein and Weinstein (\cite{EpsteinWeinstein}).
  
To study the stability of $m$-fold circles as solutions to (\ref{generalflow}) we will consider
the 
Boundary Value Problem
\begin{equation}
\label{fastcurvature0}
\left\{
\begin{array}{l}
\frac{\partial k}{\partial t}=k^2\left(k^{p-1}k_{\theta\theta}+\left(p-1\right)
k^{p-2}k_{\theta}^2+\frac{1}{p}k^{p}\right)
\quad\mbox{in} \quad\left[0,\frac{2\pi}{\lambda}\right]\times\left(0,T\right)\\
k\left(\theta,0\right)=\psi\left(\theta\right)
\qquad \mbox{on} \quad \left[0,\frac{2\pi}{\lambda}\right]
\end{array}
\right.
\end{equation}
with periodic boundary conditions, $\lambda>\sqrt{\frac{p+2}{p}}$, and $\psi$ a {\it strictly positive
function}. 
As it is (\ref{fastcurvature0}) has no immediate geometric interpretation.
However, when $\lambda$ is an appropriate rational number, (\ref{fastcurvature0}) 
is the evolution equation of the curvature of a curve being deformed via (\ref{generalflow})
under the assumption that the curve being deformed satisfies certain symmetries; more precisely,
 when $\lambda=\frac{n}{m}$, the study of equation (\ref{fastcurvature0}) is
equivalent to the study of (\ref{generalflow}) when the initial data is a perturbation
of an $m$-fold circle under a $\frac{2\pi m}{n}$-periodic perturbation.

The method we will use to prove our stability results was introduced in \cite{Cortissoz} 
(inspired by \cite{MattinglySinai}) to study the blow-up behavior of certain
nonlinear parabolic equations with periodic boundary conditions, but as the
reader will notice, it can be also used to study the stability of certain
blow-up profiles, 
and the regularity of solutions to (\ref{fastcurvature0}) 
(as a byproduct of the method we will employ, it can be shown that, under certain
conditions on the initial data, solutions to (\ref{fastcurvature0}) are analytic).
So we hope that the reader may find the method used in this paper of independent
interest.

The organization of this paper is as follows:  in Section \ref{mainresult} we present
our main result, its application to the stability problem of $m$-fold circles,
and its proof; in Section \ref{stabilitysection} we discuss the 
exponential stability of the constant steady solution of the normalized version
of (\ref{fastcurvature0}). 

\section{Main Result}
\label{mainresult}
Our results on the behavior of the $p$-curve shortening flow
will follow as a consequence of a result on the behavior of
solutions to (\ref{fastcurvature0}), that we will promptly describe; but before
we state our main result, let us set some definitions and notation.
Given $f\in L^2\left(\left[0,\frac{2\pi}{\lambda}\right]\right)$,  we write its Fourier expansion as,
\[
\sum_{n\in \mathbb{Z}} \hat{f}\left(n\right)e^{i \lambda n x},
\]
and define the family of seminorms,
\[
\left\|f\right\|_{\beta}=\max \left\{\sup_{n\neq 0}\left|n\right|^{\beta}\left|Re\left(\hat{f}\left(n\right)\right)\right|,
\sup_{n\neq 0}\left|n\right|^{\beta}\left|Im\left(\hat{f}\left(n\right)\right)\right|\right\}.
\]
As is customary, we define $C^l\left(\left[0,\frac{2\pi}{\lambda}\right]\right)$ as the space of functions with continuous
derivatives of order $l$, equipped with the norm
\[
\left\|f\right\|_{C^l\left(\left[0,\frac{2\pi}{\lambda}\right]\right)}=
\max_{j=0,1,2,\dots,l}\sup_{\theta\in\left[0,\frac{2\pi}{\lambda}\right]}
\left|\frac{d^j f \left(\theta\right)}{d\theta^j}\right|.
\]

The main result of this paper is the following theorem:
\begin{theorem}
\label{maintheorem2}
Let $\lambda > \sqrt{\frac{p+2}{p}}$.
There exists a constant $c_{p,\lambda}>0$ such that if 
\begin{equation}
\label{condition3}
\frac{\lambda}{2\pi}\int_0^{\frac{2\pi}{\lambda}} \psi\left(\theta\right)\,d\theta \geq 
c_{p,\lambda}\left\|\psi\right\|_2,
\end{equation}
then a solution to (\ref{fastcurvature0}) with initial condition $\psi$ 
is analytic (in $\theta$) and satisfies
\begin{equation}
\label{maxwell}
\left\|k\left(\theta,t\right)-\hat{k}\left(0,t\right)\right\|_{C^l\left(\left[0,\frac{2\pi}{\lambda}\right]\right)}
\leq E_{l,p,\lambda} \left(T-t\right)^{\left(\lambda^2-\frac{p+2}{p}\right)\frac{p}{p+1}},
\end{equation}
where $0<T<\infty$ is the blow-up time of the solution to (\ref{fastcurvature0}) with initial condition $\psi$. 
\end{theorem}

Let us remark that the fact that a solution to (\ref{fastcurvature0}) with initial condition $\psi>0$
blows up in finite time is a consequence of the Maximum Principle for Parabolic Equations. 
Also, the reader should notice that if $c_{p,\lambda}>0$ is large enough then there is no need to assume that $\psi$
is strictly positive, since it would be a consequence of (\ref{condition3}): this
is why even though we are assuming the positivity of $\psi$, this assumption does not appear
explicitly in the statement of the theorem.

Next we give the promised geometric corollaries of our main result.

\subsection{A stability result for $m$-fold circles}\label{stability1}
Theorem \ref{maintheorem2} has as a corollary an asymptotic (nonlinear) stability result for 
small perturbations of $m$-fold circles. Indeed, if an 
$m$-fold circle is perturbed by a $\frac{2\pi m}{n}$-periodic function, then its curvature function satisfies
equation (\ref{fastcurvature0}) with $\lambda=\frac{n}{m}$, $n$ and
$m$ mutually primes. Hence, if 
\[
\frac{n}{m}> \sqrt{\frac{p+2}{p}},
\]
and if the perturbation of the $m$-fold
circle by the $\frac{2\pi m}{n}$-periodic function  
 is such that its curvature function satisfies the hypothesis of Theorem \ref{maintheorem2}, then
it will shrink asymptotically as an $m$-fold circle under the $p$-curve shortening flow. 

Let us explain with more care. Adopting a similar notation as in \cite{Wang}, let $\gamma_m$ be the 
$m$-fold circle, i.e.,
\[
\gamma_m\left(\theta\right)=\left(\cos \theta,\sin\theta\right), \quad \theta\in\left[0,2\pi m\right],
\]
and let $\varphi$ be a smooth function of the normal angle ($\theta$) of $\gamma_m$
of period $\frac{2\pi m}{n}$ with $\frac{n}{m}>\sqrt{\frac{p+2}{p}}$. Consider the curve
\[
\gamma_{0\delta}=\gamma_m+ \delta \varphi \mathbf{N},
\]
where $\mathbf{N}$ is the outward unit normal to $\gamma_m$. It is
assumed that $\left|\delta\right|>0$ is small enough so that the curvature
of $\gamma_{0\delta}$ is strictly positive. Then we have the following result:
\begin{theorem}
If $\left|\delta\right|>0$ is small enough, then the solution to the $p$-curve shortening flow shrinks to a point
asymptotically like an $m$-fold circle. 
\end{theorem}

How small is $\left|\delta\right|$ in the previous theorem is dictated by the fact that the curvature of
$\gamma_{0\delta}$ must satisfy the hypothesis of Theorem \ref{maintheorem2}. The proof of 
this result is inmediate from Theorem \ref{maintheorem2}, and it extends the results
of Wang (\cite{Wang}), at least 
in the case when $\frac{n}{m}>\sqrt{\frac{p+2}{p}}$. Notice also that the 
condition $\frac{n}{m}>\sqrt{\frac{p+2}{p}}$ translates to 
$\frac{n}{m}>\sqrt{3}$ in the case
of the curve shortening flow which is obviously weaker than
the Abresch-Langer condition $\frac{n}{m}>\sqrt{2}$ (which is covered in Wang's work).

We can be more quantitative in describing how a solution to the normalized $p$-curve shortening
flow with initial condition $\gamma_{0\delta}$ approaches the solution
given by an $m$-fold circle. To this end we consider the following normalized version of 
(\ref{fastcurvature0})
\begin{equation}
\label{normalized}
\left\{
\begin{array}{l}
\frac{\partial \tilde{k}}{\partial \tau}=p
\tilde{k}^{1+\frac{1}{p}}\frac{\partial^2 \tilde{k}}{\partial\theta^2}+p\tilde{k}^{2+\frac{1}{p}}-
p\tilde{k}
\quad \mbox{in}\quad \left[0,\frac{2\pi}{\lambda}\right]\times \left(0,\infty\right)\\
\tilde{k}\left(\cdot,0\right)=\left(\frac{pT}{p+1}\right)^{\frac{1}{p+1}}\psi,
\end{array}
\right.
\end{equation}
with periodic boundary conditions. This normalized equation is obtained from (\ref{fastcurvature0})
by the rescaling and change of time parameter given by
\[
\tilde{k}\left(\theta,t\right)=\left(\frac{p}{p+1}\right)^{\frac{1}{p+1}}\left(T-t\right)^{\frac{1}{p+1}}k\left(\theta,t\right),
\quad \tau=-\frac{1}{p+1}\log\left(1-\frac{t}{T}\right).
\]
So we have:
\begin{corollary}
\label{expconvergence}
Let $\lambda=\frac{n}{m}$, $n$ and $m$ relatively primes, and $\lambda>\sqrt{\frac{p+2}{p}}$.
Let $\gamma_{0\delta}$ be a $\frac{2\pi}{\lambda}$-periodic perturbation of the $m$-fold circle.
There exists a $c_{p,\lambda}>0$ such that if $\left|\delta\right|$ is small enough so that
the curvature
$\psi$ of $\gamma_{0\delta}$ satisfies
\begin{equation}
\label{curvhyp}
\frac{\lambda}{2\pi}\int_0^{\frac{2\pi}{\lambda}} \psi\left(\theta\right)\,d\theta
=\frac{\lambda}{2\pi}\int_0^{\frac{2\pi}{\lambda}} \psi\left(\theta\right)\,d\theta\geq c_{p,\lambda}\left\|\psi\right\|_2,
\end{equation}
then
the curvature $\tilde{k}\left(\cdot,\tau\right)$ 
of the solution to the normalized $p$-curve shortening flow with initial condition
$\left(\frac{p+1}{pT}\right)^{\frac{1}{p+1}}\gamma_{0\delta}$ satisfies
\[
\left\|\tilde{k}\left(\theta,\tau\right)-\frac{\lambda}{2\pi}\int_{0}^{\frac{2\pi}{\lambda}}
\tilde{k}\left(\theta,\tau\right)\,d\theta \right\|_{C^l
\left[0,\frac{2\pi}{\lambda}\right]}
\leq A_{l,p,\lambda}\exp\left(-\beta\left(p,l\right)\tau\right),
\] 
where 
\[
\beta\left(p,\lambda\right)=\left(\lambda^2-\frac{p+2}{p}\right)p+1=\lambda^2p-p-1,
\]
and $A_{l,p,\lambda}>0$ is a constant, which only depends on the initial condition, and also on $p$ and $\lambda$.
Furthermore, the deformation towards a circle is through analytic curves.
\end{corollary}

The stabilization rate given by the previous corollary is optimal in the sense that
\[
\beta\left(\lambda,p\right)=\left(\lambda^2-\frac{p+2}{p}\right)p+1=\lambda^2 p -p-1, 
\]
is the smallest eigenvalue of 
the linearization around $w\equiv 1$ of the nonlinear operator
\begin{equation}
\label{ellipticpart}
pu^{1+\frac{1}{p}}\frac{\partial^2 u}{\partial \theta^2}+pu^{2+\frac{1}{p}}-pu,
\quad \theta\in\left[0,\frac{2\pi}{\lambda}\right],
\end{equation}
which corresponds to the elliptic part
of the parabolic equation obtained by normalizing (\ref{fastcurvature0}) as is 
described in the introduction, so we
cannot expect better stabilization rates for the derivatives of $k$. In Section
\ref{stabilitysection}, the last one of this paper, we also show that in the
case of the normalized flow $k\rightarrow 1$ at the rate predicted by 
the linearization of (\ref{ellipticpart}): standard methods predict that
this rate is at least $e^{-\left(\beta\left(\lambda,p\right)-\epsilon\right)\tau}$, for any $\epsilon>0$;
the nonobvious part is then to dispose of the $\epsilon$. 
The reader is advised to consult the interesting
work of Wiegner on the subject (see \cite{Wiegner}).

\subsection{Proof of Theorem \ref{maintheorem2}}
\label{theproof}
The proof of Theorem \ref{maintheorem2} will be given through a series of lemmas,
but the strategy we will follow can be described very succinctly: we shall show
that the Fourier coefficients of the solution to (\ref{fastcurvature0}) minus
its average decay
uniformly to 0 in time at the appropiate rate, when the blow-up time is approached. 
Our main tool is the analysis of the infinite dimensional ODE system satisfied
by these Fourier coefficients.

To start 
with the proof of Theorem \ref{maintheorem2},
let us introduce some notation to make our writing a bit easier.
\[
\hat{u}^{* m}\left(q_1,q_2,\dots,q_{m},t\right)
=\hat{u}\left(q_1,t\right)\hat{u}\left(q_2,t\right)\dots\hat{u}\left(q_{m},t\right),
\]
\[
H\left(p,q_1,q_2\right)=\frac{1}{p}-
\left(p-1\right)\lambda^2q_1q_2-\lambda^2q_1^2,
\]
and
\[
\mathbf{q}=\left(q_1,q_2,\dots,q_{p+1},q_{p+2}\right).
\]
Let $\mathcal{Z}$ be a finite subset of the integers  which contains 0 (i.e, $0\in\mathcal{Z}$),
and which is symmetric around 0 (i.e., if $n\in\mathcal{Z}$ then $-n\in\mathcal{Z}$).
Let
\[
\mathcal{B}_n=\left\{\left(b_1,\dots,b_{p+2}\right)\in\mathbb{Z}^{p+2}:\, 
b_{p+2}=n-b_1-b_2-\dots-b_{p+1}\right\},
\]
and define the set
\[
\mathcal{A}_n=\left\{\mathbf{b}\in \mathcal{B}_n:\,
\mbox{there are}\,\, 1\leq i<j\leq p+2 \,\,
\mbox{such that}\,\, b_i\neq 0 \,\, \mbox{and}\,\, b_j\neq 0\right\}.
\]
Consider the following finite
dimensional approximation of (\ref{fastcurvature0}) (obtained from (\ref{fastcurvature0}) 
after formally taking Fourier transform
and then restricting the infinite dimensional ODE system only to those Fourier wave numbers contained in $\mathcal{Z}$),
\begin{equation}
\label{finiteODE}
\left\{
\begin{array}{rcl}
\frac{d}{dt}\hat{k}_{\mathcal{Z}}\left(0,t\right)&=&\frac{1}{p}\hat{k}_{\mathcal{Z}}\left(0,t\right)^{p+2}+\\ 
&&
\sum_{\mathbf{q}\in\mathcal{A}_0\cap \mathcal{Z}^{p+2}}H\left(p,q_1,q_2\right)
\hat{k}_{\mathcal{Z}}^{*\left(p+2\right)}
\left(\mathbf{q},t\right),
 \\ 
&&\\
\frac{d}{dt}\hat{k}_{\mathcal{Z}}\left(n,t\right)&=&\left(\frac{p+2}{p}
-\lambda^2n^2\right)\hat{k}_{\mathcal{Z}}\left(0,t\right)^{p+1}\hat{k}_{\mathcal{Z}}\left(n,t\right) +\\ 
&&
\sum_{\mathbf{q}\in\mathcal{A}_n\cap \mathcal{Z}^{p+2}}H\left(p,q_1,q_2\right)
\hat{k}_{\mathcal{Z}}^{*\left(p+2\right)}\left(\mathbf{q},t\right), \,\,\mbox{if} \,\,n\neq 0,
\,n\in \mathcal{Z},
\end{array}
\right.
\end{equation}
with initial condition
\begin{equation}
\hat{\psi}_{\mathcal{Z}}\left(n\right)=\hat{\psi}\left(n\right), \quad \mbox{if}\quad n\in\mathcal{Z}.
\end{equation}

Notice, and this will be important but not
explicitly mentioned in our arguments, that the symmetry of $\mathcal{Z}$ guarantees that
\[
\sum_{\mathbf{q}\in \mathcal{A}_0\cap\mathcal{Z}^{p+2}}H\left(p,q_1,q_2\right)
\hat{k}_{\mathcal{Z}}^{*\left(p+2\right)}\left(q_1,\dots,q_{p+2},t\right)
\]
is real valued.

Our first lemma gives an interesting estimate on the behavior of
solutions to (\ref{finiteODE}).
\begin{lemma}[Trapping Lemma]
\label{trapping}
There exists a constant $c_{p,\lambda}>0$ independent of the choice of $\mathcal{Z}$ such that if the
initial datum $\psi$ satisfies (\ref{condition3}) then there exist a $\gamma>0$ that depends on $\psi$ such that
the solution to (\ref{finiteODE}) satisifies 
\[
\left|\hat{k}_{\mathcal{Z}}\left(n,t\right)\right|
\leq \frac{c_{p,\lambda}\hat{\psi}\left(0\right) e^{-\gamma\left|n\right|t}}{\left|n\right|^{2}}
\quad n\neq 0.
\]
\end{lemma}

\begin{proof}
Following \cite{Cortissoz}, first we fix a set $\mathcal{Z}$ as described above. 
Define the function $\hat{v}\left(n,t\right)$ as 
\[
\hat{v}\left(n,t\right):=\hat{\varphi}\left(n,t\right)\hat{k}_{\mathcal{Z}}\left(n,t\right),
\]
where,
\[
\hat{\varphi}\left(n,t\right)=e^{\gamma\left|n\right|t},
\]
with $\gamma$ small, and we let
\[
\Phi\left(q_1,\dots,q_{p+2},t\right)=\frac{\hat{\varphi}\left(q_1+\dots+q_{p+2},t\right)}
{\hat{\varphi}\left(q_1,t\right)\dots\hat{\varphi}\left(q_{p+2},t\right)}.
\]
Then we obtain the following system of ODEs, 
\begin{eqnarray}
\label{vsystem}
\notag
\frac{d}{dt}\hat{v}\left(0,t\right)&=&\frac{1}{p}\hat{v}\left(0,t\right)^{p+2} + \\\notag
&&
\sum_{\mathbf{q}\in \mathcal{A}_0\cap\mathcal{Z}^{p+2}}H\left(p,q_1,q_2\right)\Phi\left(\mathbf{q},t\right)
\hat{v}^{*\left(p+2\right)}\left(\mathbf{q},t\right)
\\
&&\\
\notag
\frac{d}{dt}\hat{v}\left(n,t\right)&=&\left(\frac{p+2}{p}+\frac{\gamma}{\hat{v}\left(0,t\right)^{p+1}}\left|n\right|
-\lambda^2n^2\right)\hat{v}\left(0,t\right)^{p+1}\hat{v}\left(n,t\right) +\\\notag
&&+
\sum_{\mathbf{q}\in \mathcal{A}_n\cap \mathcal{Z}^{p+2}}H\left(p,q_1,q_2\right)\Phi\left(\mathbf{q},t\right)
\hat{v}^{*\left(p+2\right)}\left(\mathbf{q},t\right),\notag
\end{eqnarray}
where $\mathbf{q}=\left(q_1,\dots,q_{p+2}\right)$.

Now consider the set $\Omega_{\mathcal{Z}}$
defined by,
\[
\Omega_{\mathcal{Z}}=
\left\{
w\in \mathbb{C}^{\mathcal{Z}}
:\, w\left(0\right) \geq c_{p,\lambda} 
\max_{n\in\mathcal{Z}}\left\{n^2\left|Re\left(w\left(n\right)\right)\right|,
n^2\left|Im\left(w\left(n\right)\right)\right|\right\} 
\right\}.
\]
To prove the lemma we must show that if $\psi_{\mathcal{Z}}=\left(\hat{\psi}_{\mathcal{Z}}\left(n\right)\right)_{n\in\mathcal{Z}}$ 
belongs to $\Omega_{\mathcal{Z}}$, so does 
\[
v\left(\cdot,t\right)=\left(\hat{v}\left(n,t\right)\right)_{n\in\mathcal{Z}},
\]
the solution to the ODE system (\ref{vsystem}) with initial condition 
$\psi_{\mathcal{Z}}$,
as long as it is defined ($v\left(\cdot,t\right)$, defined by the ODE system
(\ref{vsystem}), is a trajectory in $\mathbb{C}^{\mathcal{Z}}$, and
what we want to show is that once a trajectory enters $\Omega_{\mathcal{Z}}$, it never leaves).
In order to do so, we must show that whenever $v$ belongs to $\Omega_{\mathcal{Z}}$ up to time $t=\tau$, then
\[
\frac{dv}{dt}\left(\tau\right)=\left(\frac{d\hat{v}}{dt}\left(n,\tau\right)\right)_{n\in\mathcal{Z}} 
\]
points towards the interior of $\Omega_{\mathcal{Z}}$
(in this case, $\tau$ can very well be $0$, and then we have $v=\psi_{\mathcal{Z}}$, which 
by hypothesis belongs to $\Omega_{\mathcal{Z}}$).

Proving that $\frac{dv}{dt}\left(\tau\right)$ points towards the interior of $\Omega_{\mathcal{Z}}$
whenever $v\left(\tau\right)$ belongs to its boundary is a consequence of the fact that
for $c_{p,\lambda}$ conveniently chosen the following inequalities hold (see Section 2 in \cite{MattinglySinai})
\begin{eqnarray}
\label{fundamentalineq1}
\notag
&\frac{1}{p}\hat{v}\left(0,\tau\right)^{p+2}&\\
&\geq&\\
\notag
&\left|\sum_{\mathbf{q}\in \mathcal{A}_0\cap \mathcal{Z}^{p+2}}H\left(p,q_1,q_2\right)\Phi\left(q_1,\dots,q_{p+2},\tau\right)
\hat{v}^{*\left(p+2\right)}\left(q_1,\dots,q_{p+2},\tau\right)\right|&
\end{eqnarray}

\begin{eqnarray}
\label{fundamentalineq2}
&\left|\left(\frac{p+2}{p}+\frac{\gamma}{\hat{v}\left(0,\tau\right)^{p+1}}\left|n\right|
-\lambda^2n^2\right)\hat{v}\left(0,\tau\right)^{p+1}\hat{v}\left(n,\tau\right)\right|& \notag\\
&\geq&\\
&\left|
\sum_{\mathbf{q}\in\mathcal{A}_n\cap\mathcal{Z}^{p+2}}H\left(p,q_1,q_2\right)\Phi\left(q_1,\dots,q_{p+2},\tau\right)
\hat{v}^{*\left(p+2\right)}\left(q_1,\dots,q_{p+2},\tau\right)
\right|.& \notag
\end{eqnarray}

So, in what follows we will show that for a good choice of $c_{p,\lambda}$, inequality (\ref{fundamentalineq2}) holds
whenever $v$ belongs to the boundary of $\Omega_{\mathcal{Z}}$. The same reasoning can then be applied
to prove inequality (\ref{fundamentalineq1}) under the same circumstances, which
would then prove the lemma.

Now, if $v$ belongs to the boundary of $\Omega_{\mathcal{Z}}$,
then it holds that
\[
c_{p,\lambda}\max_{n\in\mathcal{Z},n\neq 0}\left\{\left|n\right|^2\left|Re \left(\hat{v}\left(n,\tau\right)\right)\right|,
\left|n\right|^2\left|Im \left(\hat{v}\left(n,\tau\right)\right)\right|\right\}
\leq \frac{\lambda}{2\pi}\int_0^{\frac{2\pi}{\lambda}}v\left(\theta,\tau\right)\,d\theta,
\]
and that there is an $n\in \mathcal{Z}$ such that
\[
c_{p,\lambda}\left\|v\right\|_2=c_{p,\lambda}\left|n\right|^2\left|Re\left(\hat{v}\left(n,\tau\right)\right)\right|=\hat{v}\left(0,\tau\right)
\]
or the same, but for the imaginary part. 
Under these assumptions, if we write 
$M=\left\|v\right\|_2$,
the righthand side of inequality (\ref{fundamentalineq2}) is bounded above by
\[
f\left(c_{p,\lambda}\right)\lambda^2 M^{p+2},
\]
where $f$ is a polynomial of degree at most $p$,
whereas the lefthand side is bounded from below (in absolute value) by
\[
c_{p,\lambda}^{p+1}\left(\lambda^2-\frac{p+2}{p n^2}-\frac{\gamma}{\hat{k}\left(0,t\right)^{p+1}\left|n\right|}\right)M^{p+2},
\]
and hence, as
\[
\lambda^2-\frac{p+2}{p n^2}-\frac{\gamma}{\hat{k}\left(0,t\right)^{p+1}\left|n\right|}>0
\]
(which can be achieved as long as $\lambda>\sqrt{\frac{p+2}{p}}$, choosing $\gamma>0$ small enough),
by taking $c_{p,\lambda}>0$ large enough, the lemma follows. The reader should have noticed also,
that a judicious choice of $c_{p,\lambda}>0$ implies that $\hat{k}\left(0,t\right)$ is increasing
(which is a consequence of inequality (\ref{fundamentalineq1})), so the choice $\gamma>0$ only depends on
$\hat{\psi}\left(0\right)$.
 
\end{proof}

\begin{remark}
The method of proof of Lemma \ref{trapping}
gives a way to estimate $c_{p,\lambda}$. Indeed, if $p=1$ (the case of the 
curve shortening flow),
 we can take
\[
c_{1,\lambda} = \frac{64\lambda^2}{\lambda^2-3}.
\]
\end{remark}

Although the Trapping Lemma is proven for finite dimensional approximations
of (\ref{fastcurvature0}), the estimate given is strong enough so
it ``passes to the limit", i.e., it holds for solutions to (\ref{fastcurvature0}),
provided that the initial condition satisfies the hypothesis of Theorem \ref{maintheorem2}
-the details are left to the reader. Hence, from now on our estimates 
are given for solutions to (\ref{fastcurvature0}), and in consequence
we drop the dependence on $\mathcal{Z}$.

Notice that the Trapping Lemma implies that solutions to (\ref{fastcurvature0}) for initial 
conditions that satisfy the
hypothesis of Theorem \ref{maintheorem2} are analytic in space. 
Also, from the Trapping Lemma, we can conclude the following useful estimate: for a solution $k$
to (\ref{fastcurvature0}) with initial condition $\psi$, which satisfies the hypothesis
of Theorem \ref{maintheorem2}, we can find constants $C,\mu>0$ such that the estimate
\[
\left|\hat{k}\left(n,t\right)\right|\leq Ce^{-\mu\left|n\right|}, \quad \mbox{for}\quad t\geq \frac{T}{2},
\]
holds. Obviously $C$ and $\mu$ may depend on the initial condition.
The interested reader can compare
this result with the work of Ferrari and Titi in \cite{FerrariTiti}, where they
prove analiticity results for certain semilinear parabolic equations in the $d$-torus.

Also, from the Trapping Lemma  and the ODE satisfied by $\hat{k}\left(0,t\right)$
we obtain the following result on the blow-up behavior of $\hat{k}\left(0,t\right)$.
\begin{lemma}[Blow-up Lemma]
\label{blowup}
Let $T>0$ be the blow-up time of a solution to (\ref{fastcurvature0}).
Under the hypothesis of Theorem \ref{maintheorem2}, for every $\eta>0$ there exists a $t_0>0$ such that
\[
\hat{k}\left(0,t\right)\geq \left(\frac{p}{p+1}\right)^{\frac{1}{p+1}}
\frac{\left(1-\eta\right)^{\frac{1}{p+1}}}{\left(T-t\right)^{\frac{1}{p+1}}}.
\]
for all $t\in\left(t_0,T\right)$.
\end{lemma}

\begin{proof} 
Using the equation satisfied by $\hat{k}\left(0,t\right)$,
by the Trapping Lemma it can be shown that $\hat{k}\left(0,t\right)$ satisfies the
differential inequality
\begin{equation*}
\frac{d}{dt}\hat{k}\left(0,t\right)\leq \frac{1}{p}\hat{k}\left(0,t\right)^{p+2}+A\hat{k}\left(0,t\right)^p,
\end{equation*}
where $A$ is a constant. Notice also that the Trapping Lemma implies that $\hat{k}\left(0,t\right)$
blows up: since $k$ blows up, and $\hat{k}\left(n,t\right)$ is conveniently bounded for $n\neq 0$,
$\hat{k}\left(0,t\right)$ must blow up. 
Hence, for every $\eta>0$, there is a $t_0>0$ such that
\begin{equation*}
\frac{d}{dt}\hat{k}\left(0,t\right)\leq \frac{1}{\left(1-\eta\right)}\frac{1}{p}\hat{k}\left(0,t\right)^{p+2}
\quad\mbox{for all}\quad t_0>0.
\end{equation*}
The Lemma follows from integrating this differential inequality.

\end{proof}

Now we can give a first estimate on the rate of decay of
the
Fourier wave numbers of solutions to (\ref{fastcurvature0}).
\begin{lemma}
\label{lambda2}
There exists $\epsilon_0>0$ which depends on $\lambda$
such that if $t>\frac{T}{2}>0$ 
then there is a constant
constant $b>0$ such that for any $0<\epsilon<\epsilon_0$,
for $n\neq 0$, the following estimate holds,
\begin{equation*}
\left|\hat{k}\left(n,t\right)\right|< b e^{-\mu \left|n\right|}\left(T-t\right)^{\epsilon}
\quad \mbox{whenever} \quad t>\frac{T}{2}.
\end{equation*}
\end{lemma}
\begin{proof} Notice that a solution to the infinite dimensional ODE system in Fourier space corresponding 
to equation (\ref{fastcurvature0}) (whose finite dimensional approximations
are described by (\ref{finiteODE})) can be written as,
\begin{eqnarray}
\label{Fourierform3}
\hat{k}\left(n,t\right)&=&
\hat{k}\left(n,\tau\right)e^{-\left(\lambda^2 n^2-\frac{p+2}{p}\right)\int_{\tau}^t \hat{k}\left(0,s\right)^{p+1}\,ds}\\ \notag
&&
+\int_{\tau}^t e^{-\left(\lambda^2 n^2-\frac{p+2}{p}\right)\int_{s}^t \hat{k}\left(0,\sigma\right)^{p+1}\,d\sigma}\times\\ \notag
&&
\quad \sum_{\mathbf{q}\in \mathcal{A}_n}
H\left(p,q_1,q_2\right)\hat{k}^{*\left(p+2\right)}\left(q_1,\dots,q_{p+2},t\right)\, ds.
\end{eqnarray}
From the Trapping Lemma, we can estimate the nonlinear term in the previous expression as
\[
\left|\sum_{\mathbf{q}\in \mathcal{A}_n}
H\left(p,q_1,q_2\right)\hat{k}^{*\left(p+2\right)}\left(q_1,\dots,q_{p+2},s\right)\right|
\leq Ce^{-\gamma\left|n\right|s}.
\]

By the Blow-up Lemma, given $\eta>0$ there is a $\delta>0$ such that if 
$t>T-\delta>\frac{T}{2}$ then
\begin{eqnarray*}
\left|\hat{k}\left(n,t\right)\right|&\leq& \left|\hat{k}\left(n,T-\delta\right)\right|
\left(\frac{T-t}{\delta}\right)^{\left(1-\eta\right)\alpha\left(\lambda,n,p\right)}+\\
&&\qquad +\left(T-t\right)^{\left(1-\eta\right)\alpha\left(\lambda,n,p\right)}e^{-\mu\left|n\right|}
\int_{T-\delta}^t\frac{1}{\left(T-s\right)^{\left(1-\eta\right)\alpha\left(\lambda,n,p\right)}}\,ds,
\end{eqnarray*}
where
\[
\alpha\left(\lambda,n,p\right)=\left(\lambda^2n^2-\frac{p+2}{p}\right)\frac{p}{p+1}.
\]
With $p$ fixed, by taking $\eta=\frac{1}{2}$, since  
$\alpha\left(\lambda,n,p\right)>0$, we obtain a 
bound,
\begin{equation*}
\left|\hat{k}\left(n,t\right)\right|\leq C\left(\left|\hat{k}\left(n,T-\delta\right)\right|\left(T-t\right)^{\epsilon}
+e^{-\mu\left|n\right|}\left(T-t\right)^{\frac{1}{2}}\right),
\end{equation*}
for any $0<\epsilon<\min\left\{\frac{1}{2}\alpha\left(\lambda,n,p\right), \frac{1}{2}\right\}$, which proves the lemma.

\end{proof}

The previous lemma already implies our stability results. However,
in order to obtain sharp estimates on the rates of uniformization of solutions to (\ref{fastcurvature0}),
and to finish the proof of Theorem \ref{maintheorem2},
Lemma \ref{blowup} does not suffice. In fact,
we can improve a bit on Lemma \ref{blowup}. So we have:
\begin{lemma}
\label{betterblowup2}
There exists a $t_0>0$ such that for all $t\in\left(t_0,T\right)$ 
the following estimate holds
\begin{equation*}
\hat{k}\left(0,t\right)\geq \left(\frac{p}{p+1}\right)^{\frac{1}{p+1}}\frac{1}{\left(\left(T-t\right)+
\left(T-t\right)^{1+\frac{2}{p+1}}\right)^{\frac{1}{p+1}}}.
\end{equation*}
\end{lemma}

\begin{proof}
As in the proof of Lemma \ref{blowup}, the following differential inequality holds
\begin{equation*}
\frac{d}{dt}\hat{k}\left(0,t\right)\leq \frac{1}{p}\hat{k}\left(0,t\right)^{p+2}+
A\hat{k}\left(0,t\right)^p,
\end{equation*}
for a constant $A>0$ independent of $t$. So using Lemma \ref{blowup}, we obtain the differential inequality
\begin{equation*}
\frac{1}{\hat{k}\left(0,t\right)^{p+2}}\frac{d}{dt}\hat{k}\left(0,t\right)\leq 
\frac{1}{p}+C\left(T-t\right)^{\frac{2}{p+1}},
\end{equation*}
which after integration gives the desired inequality.

\end{proof}

We are ready to improve the estimate on the decay of the Fourier coefficients 
of solutions to (\ref{fastcurvature0}), i.e., the estimate provided by Lemma \ref{lambda2}.
In order to proceed, we use the previous lemma to estimate the integral
\[
I=\frac{p+1}{p}\int_{\tau}^t \hat{k}\left(0,t\right)^{p+1}
\]
from below. Indeed, from Lemma \ref{betterblowup2}, since $t$ and $\tau$ are close to $T$, using Taylor's Theorem, a simple computation shows that
\[
I\geq\int_{\tau}^t \frac{1}{T-s+\left(T-s\right)^{1+\frac{2}{p+1}}}\,ds
=-\log\left(\frac{T-t}{T-\tau}\right)+O\left(1\right),
\]
and using this and (\ref{Fourierform3}) to estimate $\hat{k}\left(n,t\right)$
from above, yields
\begin{eqnarray}
\label{thebound}
\left|\hat{k}\left(n,t\right)\right|&\leq& C\left|\hat{k}\left(n,T-\delta\right)\right|
\left(\frac{T-t}{\delta}\right)^{\alpha\left(\lambda,n,p\right)}+\\ \notag
&&C \left(T-t\right)^{\alpha\left(\lambda,n,p\right)}\int_{T-\delta}^t \left(\frac{1}{T-s}\right)^{\alpha\left(\lambda,n,p\right)}
\times\\
&&
\quad \sum_{\mathbf{q}\in\mathcal{A}_n}
\left|H\left(p,q_1,q_2\right)\hat{k}^{*\left(p+2\right)}\left(q_1,\dots,q_{p+2},s\right)\right|\,ds, \notag
\end{eqnarray}
where 
\[
\alpha\left(\lambda,n,p\right)=\left(\lambda^2n^2-\frac{p+2}{p}\right)\frac{p}{p+1}.
\]
Let us now improve on the estimate given by Lemma \ref{lambda2}.
If we introduce the bound from Lemma \ref{lambda2} into (\ref{thebound}), we get
\begin{eqnarray*}
\left|\hat{k}\left(n,t\right)\right|&\leq& \left|\hat{k}\left(n,T-\delta\right)\right|
\left(\frac{T-t}{\delta}\right)^{\alpha\left(\lambda,n,p\right)}+\\
&&\qquad +\left(T-t\right)^{\alpha\left(\lambda,n,p\right)}e^{-\mu'\left|n\right|}
\int_{T-\delta}^t\frac{1}{\left(T-s\right)^{\alpha\left(\lambda,n,p\right)}}\left(T-s\right)^{2\epsilon}\,ds,
\end{eqnarray*}
with $0<\mu'<\mu$, and from which we obtain the estimate
\[
\left|\hat{k}\left(n,t\right)\right|\leq C'
e^{-\mu'\left|n\right|}\left(T-t\right)^{\min\left\{\alpha\left(\lambda,n,p\right),1+2\epsilon\right\}}.
\]
Using this new bound and plugging it into (\ref{thebound}), we improve again
our estimate on $\hat{k}\left(n,t\right)$:
\[
\left|\hat{k}\left(n,t\right)\right|\leq C''
e^{-\mu''\left|n\right|}\left(T-t\right)^{\min\left\{\alpha\left(\lambda,n,p\right),3+4\epsilon\right\}}
\quad (0<\mu''<\mu').
\]
Finally, it should be clear that if we repeat this procedure a finite number 
of times, we arrive at
\begin{equation}
\label{sharpineq}
\left|\hat{k}\left(n,t\right)\right|\leq D
e^{-\nu\left|n\right|}\left(T-t\right)^{\alpha\left(\lambda,n,p\right)},\quad n\neq 0,
\end{equation}
which in turn implies estimate (\ref{maxwell}), i.e., the conclusion of Theorem \ref{maintheorem2}.

\section{A few comments on the stabilization rate towards the constant steady state solution of the normalized 
$p$-curve shortening flow}
\label{stabilitysection}

In this section we study the exponential stability of the steady solution 
$u\equiv 1$ of the normalized $p$-curve shortening flow. Recall that
the normalized $p$-curve shortening flow is obtained from
the unnormalized $p$-curve shortening flow by the process described in section \ref{mainresult}.
To be more precise, we have that the normalized $p$-curve shortening flow is equivalent to the
Boundary Value Problem
\begin{equation}
\label{normalized}
\left\{
\begin{array}{l}
\frac{\partial u}{\partial \tau}=pu^{1+\frac{1}{p}}\frac{\partial^2 u}{\partial\theta^2}+pu^{2+\frac{1}{p}}-pu
\quad \mbox{in}\quad \left[0,\frac{2\pi}{\lambda}\right]\times \left(0,\infty\right)\\
u\left(\cdot,0\right)=\left(\frac{pT}{p+1}\right)^{\frac{1}{p+1}}\psi,
\end{array}
\right.
\end{equation}
with periodic boundary conditions. Recall that this normalized equation is obtained from (\ref{fastcurvature0})
by the rescaling and change of time parameter given by
\[
u\left(\theta,t\right)=\left(\frac{p}{p+1}\right)^{\frac{1}{p+1}}\left(T-t\right)^{\frac{1}{p+1}}k\left(\theta,t\right),
\quad \tau=-\frac{1}{p+1}\log\left(1-\frac{t}{T}\right).
\]
The reader should have noticed that Corollary \ref{expconvergence} does not give a rate of convergence of
the solution to the normalized flow towards the steady solution $u \equiv 1$. In order to
provide rates of convergence towards the steady solution, we will analyze the behavior
of $\hat{k}\left(0,t\right)$ in the case of the unnormalized equation. First observe that
the following lemma holds, and that its proof is an obvious modification of the proof given
in Lemma \ref{betterblowup2}. Again, we are under the hypothesis of
Theorem \ref{maintheorem2}.
\begin{lemma}
There is a $t_0>0$ such that if $t\in\left(t_0,T\right)$, then we have the estimate 
\[
\hat{k}\left(0,t\right)\leq 
\left(\frac{p}{p+1}\right)^{\frac{1}{p+1}}\frac{1}{\left[\left(T-t\right)-
\left(T-t\right)^{1+\frac{2}{p+1}}\right]^{\frac{1}{p+1}}}.
\]
\end{lemma}

If we let
\[
\hat{u}\left(0,t\right)=\left(\frac{p}{p+1}\right)^{\frac{1}{p+1}}\left(T-t\right)^{\frac{1}{p+1}}\hat{k}\left(0,t\right),
\]
and then compute
\[
\hat{u}\left(0,t\right)-1\leq
\frac{\left(T-t\right)^{\frac{1}{p+1}}-\left[\left(T-t\right)-\left(T-t\right)^{1+\frac{2}{p+1}}\right]^{\frac{1}{p+1}}}
{\left[\left(T-t\right)-\left(T-t\right)^{1+\frac{2}{p+1}}\right]^{\frac{1}{p+1}}},
\]
factoring out from the numerator $\left(T-t\right)^{\frac{1}{p+1}}$ and applying Taylor's Theorem we obtain
\[
\hat{u}\left(0,t\right)-1\leq \frac{\left(T-t\right)^{\frac{1}{p+1}}}{\left[\left(T-t\right)-
\left(T-t\right)^{1+\frac{2}{p+1}}\right]^{\frac{1}{p+1}}}\left(T-t\right)^{\frac{2}{p+1}}
\leq C\left(T-t\right)^{\frac{2}{p+1}}.
\]
In the same way, by using Lemma \ref{betterblowup2}, we obtain an estimate
\[
1-\hat{u}\left(0,t\right)\leq C\left(T-t\right)^{\frac{2}{p+1}}.
\]
This implies that for the normalized flow (\ref{normalized}) holds that
\[
\left|\frac{\lambda}{2\pi}\int_0^{\frac{2\pi}{\lambda}}u\left(\theta,\tau\right)\,d\theta-1\right|\leq C\exp\left(-2\tau\right).
\]
Using this estimate and Corollary \ref{expconvergence}, we obtain the estimate
\begin{equation}
\label{stability}
\left\|u\left(\theta,\tau\right)-1\right\|_{L^{\infty}\left(\left[0,\frac{2\pi}{\lambda}\right]\right)}
\leq C\exp\left(-\omega \tau\right),
\end{equation}
where
\[
\omega=\min\left\{\lambda^2p-p-1, 2\right\}.
\]
Notice that 
\[
\lambda^2p-p-1 \leq 2
\]
whenever
\[
\sqrt{\frac{p+2}{p}}<\lambda\leq \sqrt{\frac{p+3}{p}},
\]
so it is in this case that we obtain a rate of decay towards the steady state
corresponding to the first eigenvalue of the
elliptic part of the normalized $p$-curve shortening flow. 
However, if we use estimate (\ref{sharpineq}) in the proof of Lemma \ref{betterblowup2}, 
we can bound
\[
\left|\sum_{\mathbf{q}\in\mathcal{A}_0}H\left(p,q_1,q_2\right)
\hat{k}^{*\left(p+2\right)}
\left(\mathbf{q},t\right)\right|\leq C\left(T-t\right)^{2\alpha\left(\lambda,p\right)}\hat{k}\left(0,t\right)^p
\]
with
\[
\alpha\left(\lambda,p\right)=\left(\lambda^2-\frac{p+2}{p}\right)\frac{p}{p+1},
\]
then from the ODE satisfied 
by $\hat{k}\left(0,t\right)$, we can
obtain the improved bound from below
\[
\hat{k}\left(0,t\right)\geq 
\left(\frac{p}{p+1}\right)^{\frac{1}{p+1}}\frac{1}{\left[\left(T-t\right)+
\left(T-t\right)^{2\alpha\left(\lambda,p\right)+\frac{2}{p+1}+1}\right]^{\frac{1}{p+1}}},
\]
and also the corresponding bound from above. Hence, proceeding as we just did, we arrive at an estimate
\[
\left|\frac{\lambda}{2\pi}\int_0^{\frac{2\pi}{\lambda}}u\left(\theta,\tau\right)\,d\theta-1\right|\leq 
C\exp\left(-\left(2\lambda^2p-p\right)\tau\right).
\]
from which follows that (\ref{stability}) holds now for all $\lambda>\sqrt{\frac{p+2}{p}}$ with
\[
\omega=\min\left\{\lambda^2p-p-1, 2\lambda^2p-p\right\}=\lambda^2 p -p -1.
\]

So we finish with the following result.
\begin{proposition}
Let $\lambda>\sqrt{\frac{p+2}{p}}$, let $\psi>0$ and assume 
that it satisfies the hypothesis of Theorem \ref{maintheorem2}.
Then the solution $u\left(\theta,\tau\right)$ to (\ref{normalized}) 
corresponding to the solution to (\ref{fastcurvature0})
with $\psi$ as initial data, satisfies an estimate
\[
\left\|u\left(\theta,\tau\right)-1\right\|_{L^{\infty}\left[0,\frac{2\pi}{\lambda}\right]}
\leq C\exp\left(-\left(\lambda^2p-p-1\right)\tau\right),
\]
where $C>0$ is a constant that depends on $\psi$, $p$ and  $\lambda$.
\end{proposition}


\begin{thebibliography}{}

\bibitem[1]{AbreschLanger} U. Abresch and J. Langer, 
The normalized curve shortening flow and homothetic solutions.
J. Differential Geom. {\bf 23} (1986), no. 2, 175--196. 

\bibitem[2]{Andrews} B. Andrews, Evolving convex curves. Calc. Var. Partial Differential Equations 
{\bf 7} (1998), no. 4, 315--371.

\bibitem[3]{Andrews2} B. Andrews, Classification of limiting shapes for isotropic curve flows. 
J. Amer. Math. Soc. {\bf 16} (2003), no. 2, 443--459.

\bibitem[4]{Cortissoz} J. Cortissoz, On the blow-up behavior of
a nonlinear parabolic equation with periodic boundary conditions. Arch. Math. (Basel) {\bf 97}
(2011), 69--78.

\bibitem[5]{EpsteinWeinstein}C. L. Epstein and M. I. Weinstein,
A stable manifold theorem for the curve shortening equation. Commun. Pure Appl. Math. 40 (1987), 119--139.

\bibitem[6]{FerrariTiti} A.B. Ferrari, E. S. Titi, Gevrey regularity for nonlinear parabolic equations.
Comm. Partial Differential Equations {\bf 23} (1998), no.1--2, 1--16. 

\bibitem[7]{GageHamilton}M. Gage and R.S. Hamilton, The heat equation shrinking convex plane curves.
J. Differential Geom. {\bf 23} (1986), no. 1, 69--96.

\bibitem[8]{Huang} R. L. Huang, Blow-up rates for the general curve shortening flow. 
J. Math. Anal. Appl. {\bf 383} (2011), no 2, 482--487.

\bibitem[9]{MattinglySinai} J. Mattingly and Ya. Sinai, An elementary 
proof of the existence and uniqueness theorem for the Navier-Stokes equations. 
Commun. Contemp. Math. {\bf 1} (1999), no. 4, 497--516.

\bibitem[10]{Sesum} N. Sesum, Rate of convergence of the mean curvature flow. Comm. Pure Appl. Math. {\bf 61}(2008),
no 4, 464--485.

\bibitem[11]{Wang} X.--L. Wang, The stability of m-fold circles in the curve shortening problem. 
Manuscripta Math. {\bf 134} (2011), no. 3--4, 493--511.

\bibitem[12]{Wiegner} M. Wiegner, On the asymptotic behaviour 
of solutions of nonlinear parabolic equations. Math. Z. {\bf 188} (1984), no. 1, 3--22.

\bibitem[13]{Winkler} M. Winkler, Blow-up of solutions to a degenerate parabolic equation not in divergence form.  
J. Differential Equations {\bf 192} (2003), no. 2, 445--474. 
\end{thebibliography}
\end{document}